\documentclass{amsart}
\usepackage{amsmath, amsthm, amssymb, amscd, mathrsfs, eucal, epsfig}

\input xy
\xyoption{all}

\usepackage{hyperref}

\begin{document}

\newtheorem{theorem}{Theorem}[subsection]
\newtheorem{lemma}[theorem]{Lemma}
\newtheorem{corollary}[theorem]{Corollary}
\newtheorem{conjecture}[theorem]{Conjecture}
\newtheorem{proposition}[theorem]{Proposition}
\newtheorem{question}[theorem]{Question}
\newtheorem{problem}[theorem]{Problem}
\newtheorem*{main_thm}{Random Surface Subgroup Theorem~\ref{thm:random_amalgam_surface}}
\newtheorem*{claim}{Claim}
\newtheorem*{criterion}{Criterion}
\theoremstyle{definition}
\newtheorem{definition}[theorem]{Definition}
\newtheorem{construction}[theorem]{Construction}
\newtheorem{notation}[theorem]{Notation}
\newtheorem{convention}[theorem]{Convention}
\newtheorem*{warning}{Warning}

\theoremstyle{remark}
\newtheorem{remark}[theorem]{Remark}
\newtheorem{example}[theorem]{Example}
\newtheorem*{case}{Case}

\def\id{\text{id}}
\def\Id{\text{Id}}
\def\1{{\bf{1}}}
\def\p{{\mathfrak{p}}}
\def\H{\mathbb H}
\def\Z{\mathbb Z}
\def\R{\mathbb R}
\def\C{\mathbb C}
\def\F{\mathbb F}
\def\P{\mathbb P}
\def\Q{\mathbb Q}
\def\E{{\mathcal E}}

\def\tra{\textnormal{tr}}
\def\length{\textnormal{length}}

\newcommand{\marginal}[1]{\marginpar{\tiny #1}}

\title{Random graphs of free groups contain surface subgroups}
\author{Danny Calegari}
\address{Department of Mathematics \\ University of Chicago \\
Chicago, Illinois, 60637}
\email{dannyc@math.uchicago.edu}
\author{Henry Wilton}
\address{Department of Mathematics \\ University College London \\ Gower Street \\ London \\ WC1E 6BT \\ UK}
\email{hwilton@math.ucl.ac.uk}

\date{\today}

\begin{abstract}
A random graph of free groups contains a surface subgroup.
\end{abstract}

\maketitle

\section{Introduction}

Gromov's Surface Subgroup Question asks whether every one-ended hyperbolic group contains
a subgroup isomorphic to the fundamental group of a closed surface with $\chi<0$.

In this paper we show that a {\em random graph of free groups} contains
{\em many} closed surface subgroups, with probability going to $1$ as a certain parameter
in the model of randomness goes to infinity. 

Graphs of free groups are a very important special case for Gromov's question, for 
several reasons. For example, let $G$ be a one-ended hyperbolic group which is 
isomorphic to the fundamental group of a non-positively curved cube complex (informally, 
$G$ is said to be {\em cubulated}).  Agol \cite{Agol_VHC} showed that $G$ is virtually 
special, and therefore $G$ contains a subgroup isomorphic to a one-ended graph of free 
groups (see Theorem \ref{thm: GoGs inside special groups}).  Many classes of hyperbolic 
groups are known to be cubulated, including
\begin{enumerate}
\item{$C'(1/6)$ groups (Wise, \cite{Wise_cancel});}
\item{random groups at density $<1/6$ (Ollivier--Wise, \cite{Ollivier_Wise});}
\end{enumerate}
and others. Such groups are all now known to contain graphs of free groups, so our
main result makes it plausible that they all contain surface subgroups.

\subsection{Precise statement of main theorem}

Let $F_k$ and $F_l$ be free groups of rank $k$ and $l$, and suppose we have chosen a
free generating set for each group. A {\em random homomorphism of length $n$} is a
homomorphism $\phi:F_k \to F_l$ which takes each of the generators of $F_k$ to 
a reduced word of length $n$ in the generators of $F_l$ (and their inverses), independently
and randomly with the uniform distribution. See \S~\ref{subsection:random_homomorphisms} 
for more details. 

If we fix a finite graph, and for each edge and vertex in the graph we fix a free group of
finite (nonzero) rank and a free generating set, we can define a {\em random} graph 
of free groups (with the given edge and vertex groups and the given generators) 
by taking each homomorphism of an edge group to a vertex group to be a random homomorphism
of length $n$. Informally we call the result a {\em random graph of free groups of
length $n$}. Evidently, to show that a random (nontrivial) graph of free groups contains
a surface subgroup, it suffices to prove it in the case of an amalgamated free product
or an HNN extension. With this terminology, our main theorem is the following:

\begin{main_thm}
Let $H:=F_1 *_G F_2$ or $H:=F *_G$ be obtained by amalgamating two free groups over random
subgroups of rank $k\ge 1$ of length $n$, or by taking an HNN extension over two random
subgroups of rank $k\ge 1$ of length $n$. Then $H$ contains a closed surface subgroup
with probability $1-O(e^{-Cn})$.
\end{main_thm}

Note that for any $\lambda>0$ a random graph of free groups of length $n$ 
satisfies the small cancellation condition $C'(\lambda)$ with probability $1-O(e^{-Cn})$
(this echoes an observation made by Button \cite{Button}) and therefore random graphs
of free groups are hyperbolic and virtually special.

The proof of the main theorem is {\em constructive}; that is, given a graph of groups
there is an explicit procedure (guaranteed to work with very high probability) 
to construct a surface subgroup, which is certified as injective by local combinatorial
conditions. A major step in the construction depends on being able to build a
{\em folded fatgraph} with prescribed boundary satisfying certain equidistribution
properties (namely $(T,\epsilon)$-pseudorandomness --- 
see \S~\ref{subsection:random_homomorphisms}); this step is
carried out in Calegari--Walker \cite{Calegari_Walker_LP}, Thm.~8.9. and was used there
to construct injective surface subgroups in certain {\em ascending} HNN extensions of
free groups. The methods from \cite{Calegari_Walker_LP} can easily be adapted to 
certify the existence of surface subgroups in particular graphs of free
groups using the program {\tt scallop} \cite{Walker_scallop}. Experiments suggest that 
such surface subgroups are extremely easy to find. The proof of the main theorem also
produces not one but infinitely many surface subgroups; see 
Remark~\ref{remark:many_subgroups}.

\section{Free groups}

\subsection{Standard rose}

Let $F$ be a finitely generated free group. We fix an identification of $F$ with
$\pi_1(X)$ where $X$ is a {\em rose} --- i.e.\/ a wedge of finitely many circles.
Informally, we say $X$ is a {\em rose for $F$}.
The oriented circles of $X$ determine a (free) generating set for $F$ which we denote
$a,b,c,\cdots$. Inverses are denoted by upper case letters, so $A:=a^{-1}$, $B:=b^{-1}$ and
so on.

\begin{definition}
A graph $Y$ {\em over $X$} is a graph together with a map $f:Y \to X$ taking edges of $Y$
to reduced simplicial paths in $X$. A graph $Y$ over $X$ is {\em folded} if the map $f$
is an immersion (i.e.\/ if it is locally injective).
\end{definition}

If $Y$ is a graph over $X$, each oriented edge of $Y$ is labeled with a reduced word in $F$
in such a way that reversing the orientation gives the inverse label. If $Y$ is folded,
immersed loops in $Y$ are labeled with cyclically reduced words in $F$.

\subsection{Core associated to a subgroup}

Let $G$ be a finitely generated subgroup of $F$, and let $X_G$ be the cover of $X$ associated
to $G$. We think of $X_G$ as a graph with a basepoint.

There is a compact core $Y_G \subset X_G$, defined to be the minimal subgraph of $X_G$
containing the basepoint so that the inclusion $Y_G \to X_G$ is a homotopy equivalence.
We think of $Y_G$ as a graph with a basepoint.

Stallings showed how to obtain $Y_G$ algorithmically by starting with a rose whose edges
are labeled by reduced words in $F$ (the generators of $G$) and then folding the rose
until it immerses in $X$ \cite{stallings_topology_1983}.

\subsection{Core associated to a conjugacy class of subgroup}

Let $Z_G \subset Y_G$ be the minimal subgraph of $Y_G$ so that the inclusion $Z_G \to Y_G$
is a homotopy equivalence. Informally, $Y_G$ is obtained from $Z_G$ by connecting it
to the basepoint  (in $X_G$). We think of $Z_G$ as a graph without a basepoint.

The graph $Z_G$ depends only on the conjugacy class of $G$ in $F$.

\section{Amalgams and HNN extensions}

In this section, we recall the standard construction of an Eilenberg--Mac~Lane space a 
for graph of free groups, and give a criterion for a map from a surface to be $\pi_1$-injective.

\subsection{Mapping cylinder}

Let $F$ be a free group, and let $G$ be a finitely generated subgroup. If $X$
is a rose for $F$, and $Z_G$ is the core graph associated to the conjugacy class of $G$,
there is an immersion $f:Z_G \to X$ and we can build the mapping cylinder
$$C_f:= Z_G \times [0,1] \cup X / (z,1) \sim f(z)$$

More generally, let $G_1,\ldots,G_n$ be a finite collection of finitely generated subgroups 
of $F$ and, for each $i$, let $f_i:Z_{G_i}\to X$ be the corresponding immersion  of core graphs.  
Then we may consider the coproduct immersion
\[
f=\coprod_i f_i:\coprod_iZ_{G_i}\to X
\]
and build the mapping cylinder $C_f$ in the same manner.

\subsection{Amalgams}

Let $F_1$ and $F_2$ be free groups with roses $X_1$ and $X_2$, and let $G$ be a finitely
generated free group with inclusions $\phi_i:G \to F_i$ so that we can form the amalgamated free
product $$H:=F_1 *_G F_2$$
There are core graphs $Z_{G,1}$ and $Z_{G,2}$ associated to $G$, and immersions
$f_i:Z_{G,i} \to X_i$ for $i=1,2$ giving rise to mapping cylinders $C_{f_1}$ and
$C_{f_2}$. The map $\phi=\phi_2\circ\phi_1^{-1}$ gives rise to a canonical homotopy class of 
homotopy equivalence $Z_{G,1} \to Z_{G,2}$.

Let $W_\phi$ be the space obtained from the mapping cylinders $C_{G,i}$ by gluing $Z_{G,1}$
to $Z_{G,2}$ by a homotopy equivalence representing $\phi$. Then $\pi_1(W_\phi)=H$, and $W_\phi$ contains
subgraphs $Z$, $X_1$, $X_2$ with fundamental groups corresponding to the subgroups 
$G$, $F_1$, $F_2$.

\subsection{HNN extensions}

Let $F$ be a free group with rose $X$, and let $G$ be a finitely generated free group with 
two inclusions $\phi_1,\phi_2:G\to F$.  Let $\phi=\phi_2\circ\phi_1^{-1}$.  We can form 
the HNN extension
\[
H:=F*_{\phi}
\]
similarly.  Again, there are core graphs  $Z_{G,1}$ and $Z_{G,2}$, where $Z_{G,i}$ is the 
core of the covering space of $X$ associated to the conjugacy class of $\phi_i(G)$.  These 
are equipped with immersions $f_i:Z_{G,i}\to X$ and the coproduct $f=f_1\sqcup f_2$ defines 
a mapping cylinder $C_f$.

In this case, we let $W_\phi$ be the space obtained from $C_f$ by identifying the two copies 
of $Z_{G,i}$ by a homotopy equivalence representing $\phi$.  As before, $\pi_1(W_\phi)=H$, 
and $W_\phi$ contains subgraphs $Z$ and $X$ with fundamental groups corresponding to the 
subgroups $G$ and $F$, respectively.

\begin{remark}
Of course, one can similarly construct Eilenberg--Mac~Lane spaces for any graph of free groups.  
However, as the fundamental group of a graph of free groups contains an amalgamated product or 
HNN extension as a subgroup, we can restrict ourselves to these cases without loss of generality.
\end{remark}

\subsection{Maps of surfaces}

In this section we state a criterion for a continuous map from a surface into a graph of spaces 
to be injective.

Let $W_\phi$ be one of the spaces constructed above and let $\sigma:S\to W_\phi$ be a 
continuous map from a closed, oriented surface to $W_\phi$, transverse to $Z$.  
Let $\alpha = \sigma^{-1}(Z) \subset S$.  Cutting along $\alpha$ decomposes $S$ into a finite set 
of compact subsurfaces with boundary $S_1,\ldots, S_n$.  Each comes equipped with a map of pairs
\[
\sigma_i: (S_i,\partial S_i)\to (C_f,Z)
\]
where $f:Z\to X$ is an immersion of graphs and $C_f$ is the corresponding mapping cylinder.  
Furthermore, there is an orientation-reversing involution of the disjoint union of the boundaries 
of the $S_i$, coming from how they were glued up in $S$ along $\alpha$.

\begin{definition}
Consider a map of pairs
\[
\sigma_i:(S_i,\partial S_i)\to (C_f,Z)
\]
where $S$ is a compact surface with boundary.  A {\em compressing bigon} for $\sigma_i$ is the 
continuous image of a bigon $B\to C_f$ with $\partial B$ equal to the union of two arcs 
$\alpha\cup\beta$ such that:
\begin{enumerate}
\item there is a proper essential embedding of $\alpha$  in $S_i$ so that the restriction of 
$B\to C_f$ to $\alpha$ equals $\sigma_i$; and
\item $\beta$ is mapped into $Z$.
\end{enumerate}
\end{definition}

\begin{lemma}\label{lemma:compressing_bigon_exists}
If $\sigma$ is not $\pi_1$-injective then some $\sigma_i$ admits a compressing bigon.
\end{lemma}
\begin{proof}
Let $\gamma$ be an immersed essential loop in $S$ mapping to an inessential loop in $W_\phi$. 
Suppose that $\gamma$ is chosen transverse to $\alpha$, and intersecting it in the least 
number of components (so that every arc of intersection is essential in some $S_i$).  
Let $D$ be a disk with $\partial D=\gamma$, and let $\sigma:D \to W_\phi$ extend $\sigma(\gamma)$.   
Make $D$ transverse to $Z$, and remove loops of intersection by a homotopy. 

An outermost bigon in $D-\sigma^{-1}(Z)$ has one arc in $\gamma$ and the other mapping to $Z$.
Since every arc of $\gamma$ is essential and proper in $S_i$, it is a compressing bigon.
\end{proof}

\begin{lemma}\label{lemma:bigon_gives_loop_in_G}
Suppose $S_i$ admits a compressing bigon. Then there is an essential non-boundary parallel
loop in $S_i$ whose image under $\sigma_i$ is freely homotopic into $Z$.
\end{lemma}
\begin{proof}
Let $B$ be a compressing bigon, let $\alpha$ be the arc of $\partial B$ proper in $S_i$,
and let $\beta$ be the other arc, mapping to $Z$. 

If the two vertices of $B$ are on the
same component $\delta$ of $\partial S_i$, let $\delta'$ be a subarc of $\partial S_i$ joining these
vertices. Then, for any $n$, the concatenation $\delta^n \delta' \alpha$ is a loop in $S_i$ homotopic to
$\delta^n \delta' \beta$ which is in $Z$, so $\delta^n \delta' \alpha$ is conjugate into $G$.  
Moreover, $\delta^n \delta' \alpha$ is essential, and is not boundary parallel in $S_i$ for all 
but at most finitely many values of $n$, or else $\alpha$ would be properly homotopic into $\delta$.

If the two vertices of $B$ are on different components $\delta$, $\delta'$ of $\partial S_i$,
for any nonzero $n,m$ the loop $\delta^n \alpha (\delta')^m \alpha^{-1}$ is essential in $S_i$ 
and homotopic into $Z$, and for all but finitely many $n,m$ it is not homotopic into $\partial S_i$.
\end{proof}

\section{Folded fatgraphs}

\subsection{Fatgraphs}

\begin{definition}
A {\em fatgraph} is a graph $Y$ together with a cyclic order on the edge incident to each vertex.
\end{definition}

A fatgraph admits a canonical {\em fattening} to a compact, oriented surface $S(Y)$ in such a way 
that $Y$ includes in $S(Y)$ as a spine, and there is a canonical deformation retraction 
of $S(Y)$ to $Y$. 

Pulling back the simplicial structure of $Y$ gives $\partial S(Y)$ the structure of a graph.

\begin{definition}
A {\em fatgraph over $X$} is a fatgraph $Y$ whose underlying graph is a graph over $X$.  
A fatgraph $Y$ over $X$ is {\em folded} if the underlying graph is folded.
\end{definition}

Suppose $Y$ is a folded fatgraph over $X$. Then the composition $S(Y) \to Y \to X$ is an 
injection on $\pi_1$.

\subsection{Relative fatgraphs}

Let $F$ be a free group, and suppose we have fixed a rose $X$ with $\pi_1(X)=F$.  
Let $f_i:Z_i\to X$ be finitely many immersions of finite graphs and let $f:Z\to X$ be their coproduct.  
Equivalently, we may fix a finite set of subgroups $G_1,\ldots, G_n$ and set $Z$ to be the disjoint 
union of the corresponding core graphs $Z_{G_i}$.

\begin{definition}
A folded fatgraph $Y$ over $X$ has {\em boundary in $Z$} if for every component $\partial_i$ of 
$\partial S(Y)$ the image of this loop in $\pi_1(X)$ lifts to $Z$.
\end{definition}

The immersion of $Y$ into $X$ extends naturally to a map $\sigma$ of $S(Y)$ into the mapping 
cylinder $C_f$ such that the boundary $\partial S(Y)$ immerses into the natural copy of $Z$.  
That is, we have the following commutative diagram
\[\xymatrix{
    Y \ar@{>}[r]\ar@{>}[d] & S(Y)\ar@{>}[d]^\sigma & \partial S(Y)\ar@{>}[l]\ar@{>}[d]\\
     X\ar@{>}[r] & C_f & Z\ar@{>}[l]^f
}\]
where the inclusions $Y\to S(Y)$ and $X\to C_f$ are both deformation retracts.

\begin{definition}
A folded fatgraph $Y$ over $X$ with boundary in $Z$ is {\em boundary incompressible} if every 
essential loop in $S(Y)$ whose image lifts to $Z$ is already freely homotopic (in $S(Y)$) into 
$\partial S(Y)$.
\end{definition}

It is evident that the condition of being boundary incompressible rules out the  existence of a 
compressing bigon, by Lemma~\ref{lemma:bigon_gives_loop_in_G}.  Our next result strengthens this 
criterion.  To state the result cleanly, we will make use of the \emph{fibre product}.

Let $p_i:\widehat{C}_i\to C_f$ be the covering space of the mapping cylinder $C_f$ corresponding 
to the conjugacy class of the subgroup $G_i$.  The covering space $\widehat{C}_i$ can be 
constructed as follows: if $\widehat{X}_i$ is the covering space of $X$ corresponding to 
the conjugacy class of $G_i$ then the immersion $f_i:Z_i\to X$ lifts to an embedding 
$\hat{f}_i:Z_i\to \widehat{X}_i$, which identifies $Z_i$ with the core of $\widehat{X}_i$; 
$\widehat{C}_i$ is the  mapping cylinder of $\hat{f}_i$.

\begin{definition}
The space
\[
S(Y)\times_{C_f} \widehat{C}_{i}=\{(x,y)\in S(Y)\times \widehat{C}_{i}\mid \sigma(x)=p_i(y)\}~,
\]
is called the \emph{fibre product} of the maps $\sigma$ and $p_i$.
\end{definition}

It is a standard exercise to check that the projection $S(Y)\times_{C_f} \widehat{C}_i\to S(Y)$ 
is a covering map.

\begin{proposition}\label{prop:boundary_compressible_in_G}
Let $Y$ be a folded fatgraph with boundary in $Z$. If $S(Y)$ is boundary compressible then there is:
\begin{enumerate}
\item a homotopically non-trivial loop $\gamma$ in the boundary of $S(Y)$,
\item an essential, non-boundary-parallel loop $\gamma'$ in $S(Y)$, and
\item a connected component $\widehat{S}(Y)$ of one of the fibre products $S(Y)\times_{C_f} \widehat{C}_i$,
\end{enumerate}
such that both $\gamma$ and $\gamma'$ lift to $\widehat{S}(Y)$.
\end{proposition}
\begin{proof}
Suppose $S(Y)$ is boundary compressible, so there is some compressing bigon $B$
as in Lemma~\ref{lemma:bigon_gives_loop_in_G} with one boundary arc $\alpha$ proper in
$S$.  We will give a proof in the case that both endpoints of $\alpha$ lie on the same boundary 
component $\gamma$; the proof in the other case is similar.

Fix a base point $*\in\gamma\cap\alpha$.  We may then take $\sigma(*)$ as a compatible base point 
in some $Z_i$, so $\pi_1(Z_i,\sigma(*))$ becomes a natural choice of representative for $G_i$ 
inside its conjugacy class.  This in turn defines a base point $\hat{*}$ in the covering space 
$\widehat{C}_i$.  Since $\sigma(*)=p(\hat{*})$, the pair $(*,\hat{*})$ defines a point in the 
fibre product $S(Y)\times_{C_f} \widehat{C}_i$.  Let $\widehat{S}(Y)$ be the component of 
$S(Y)\times_{C_f} \widehat{C}_i$ that contains $(*,\hat{*})$. It follows from standard 
covering-space theory that
\[
\pi_1(\widehat{S}(Y),(*,\hat{*}))=\sigma_*^{-1}\pi_1(\widehat{C}_i,\hat{*})
\]
when thought of as a subgroup of $\pi_1(S(Y),*)$.  In particular, $\gamma$ lifts to $\widehat{S}(Y)$.

Let $\delta$ be the subarc of $\gamma$ with the same endpoints as $\alpha$ and let 
$\gamma''=\delta\cup\alpha$.  Since $\delta\subseteq Z_G$ and $\alpha$  bounds one half of a 
bigon whose other side is in $G$, we see that $\sigma(\gamma'')$ is contained in $G$, and so 
$\gamma''$ also lifts to $\widehat{S}(Y)$.  Indeed, the entire subgroup of $\pi_1(S(Y),*)$ 
generated by $\gamma$ and $\gamma''$ also lifts to $\widehat{S}(Y)$.

Because $\gamma''$ is not homotopic into $\gamma$, this subgroup is free of rank two, and so 
contains elements represented by curves which are essential and not boundary parallel.  
Let $\gamma'$ be any such curve.
\end{proof}

\section{Random subgroups}

\subsection{Random homomorphisms}\label{subsection:random_homomorphisms}

Fix integers $k,l\ge 2$ and let $F_k$, $F_l$ be free groups on $k$, $l$ generators
respectively.

\begin{definition}
A {\em random homomorphism of length $n$} is a homomorphism $\phi:F_k \to F_l$ 
which takes each generator of $F_k$ to a random reduced word in $F_l$ of length $n$, 
independently and with the uniform distribution.
\end{definition}

If we fix $F$ free and $X$ a rose for $F$ (and therefore, implicitly, a free generating
set for $F$) then if we pick $k$, it makes sense to define a {\em random $k$-generator
subgroup of $F$ of length $n$} to be the image of $F_k \to F$ under a random 
homomorphism of length $n$. We suppose $G \subseteq F$ is a random $k$-generator
subgroup of length $n$.

The construction of a random subgroup chooses for us a generating set for $G$.
The core graph $Y_G$ is obtained from a rose $R_G$ for $G$ (with edges corresponding
to the given generating set) by writing new edge labels which are reduced words in $F$,
and then folding the result. 

\begin{lemma}\label{lem:random_folding}
Let $\phi:F_k \to F_l$ be a random homomorphism of length $n$. Then
with probability $1-O(e^{-Cn})$ the result of folding the rose $R_G \to Y_G$
is a homotopy equivalence, and is an isomorphism away from a neighborhood
of the vertex of $R_G$ of diameter $O(\log{n})$.
\end{lemma}
\begin{proof}
This lemma is simply the observation that two random reduced words in $F_l$
of length $n$ have a common prefix or suffix of length at most $O(\log{n})$ with
probability $1-O(e^{-Cn})$.
\end{proof}

\begin{definition}\label{def:pseudorandom}
Fix a free group $F$ of rank $l$ and a free generating set. A cyclically reduced word $w$ in the
generators is $(T,\epsilon)$-pseudorandom if for every reduced word $\sigma$ in $F$
of length $T$, the number of copies of $\sigma$ in $w$ (denoted $C_\sigma(w)$) satisfies
$$1-\epsilon \le \frac {C_\sigma(w)} {\length(w)}\cdot (2l)(2l-1)^{T-1} \le 1+\epsilon$$
\end{definition}

In a free group of rank $l$ there are $(2l)(2l-1)^{T-1}$ reduced words of length $T$.
So informally, a cyclically reduced word is $(T,\epsilon)$-pseudorandom if its distribution
of subwords of length $T$ is distributed as in a random word, up to an error of order 
$\epsilon$.

We also extend the definition of pseudorandom to finite collections of cyclically reduced
elements. Moreover, by abuse of notation, if we fix a rose $X$ for $F$, we say that
an immersed $1$-manifold $\Gamma \to X$ is $(T,\epsilon)$-pseudorandom if the corresponding
collection of cyclic words in $F$ is $(T,\epsilon)$-pseudorandom.

\begin{lemma}\label{lem:random_is_pseudorandom}
Fix any positive integer $T$, and $\epsilon > 0$. 
Then if $\phi:F_k \to F_l$ is a random homomorphism of length $n$, 
with probability $1-O(e^{-Cn})$ the cyclically reduced
representative of $\phi(g)$ is $(T,\epsilon)$-pseudorandom for every nontrivial $g$ in $F_k$.
\end{lemma}
\begin{proof}
A random word in $F_l$ of length $n$ is $(T,\epsilon)$-pseudorandom with probability
$1-O(e^{-Cn})$. Each generator of $F_k$ gets taken by $\phi$ to a random word in $F_l$ of
length $n$, and non-inverse generators get taken to words with at most $O(\log{n})$
cancellation, with probability $1-O(e^{-Cn})$. The proof follows.
\end{proof}

\subsection{Combinatorial Rigidity}

A reduced word in $F$ is represented by an immersed path $\delta:[0,1] \to X$. A cyclically
reduced word in $F$ is represented by an immersed loop $\gamma:S^1 \to X$. Every element
of $F$ is uniquely represented by a reduced word, and every conjugacy class is uniquely
represented by a cyclically reduced word.

\begin{definition}
Let $f:Z\to X$ be an immersion of finite graphs.  A loop $\gamma:S^1\to X$ is 
{\em combinatorially rigid in $Z$} if it admits a {\em unique} lift $\gamma:S^1 \to Z$.

Let $G_1,\ldots, G_n$ be a finite collection of subgroups of $F$, represented by an immersion 
$f:Z\to X$.  A conjugacy class represented by an element $g$ in $F$ is 
\emph{combinatorially rigid in the $G_i$} if the unique geodesic loop $\gamma:S^1\to X$ that 
represents it is combinatorially rigid in $Z$.  It is {\em fully combinatorially rigid in the $G_i$} 
if the conjugacy classes of $g$ and all its (nontrivial) powers are combinatorially rigid in the $G_i$.
\end{definition}

Note that a loop $\gamma$ which is (fully) combinatorially rigid in the $G_i$ is necessarily 
conjugate into a unique $G_i$, by definition.  In fact, combinatorial rigidity corresponds to 
a well known algebraic condition.

\begin{definition}\label{def:malnormal}
Let $G$ be a group. A subgroup $H$ of $G$ is {\em malnormal} if $H \cap H^g = 1$ for every 
$g \in G-H$. More generally, a family of subgroups $\{H_i\}$ is \emph{malnormal} if 
$H_i \cap H_j^g = 1$ whenever $i\neq j$ or $g \in G-H_i$.
\end{definition}

\begin{proposition}\label{prop:malnormal_rigid}
Let $\{G_i\}$ be a family of subgroups of $F$. Then $\{G_i\}$ is malnormal if and only if 
for every $j$, every nontrivial conjugacy class in $G_j$ is fully combinatorially rigid in the $G_i$.
\end{proposition}
\begin{proof}
First, suppose that the $G_i$ are malnormal, and let $f:Z\to X$ be the corresponding immersion of 
finite graphs.  Fix a base point on the image of $\gamma$ and let $g$ be the corresponding  
element of $G_j$.  Replacing $g$ by $g^n$, it suffices to show that $\gamma$ is 
combinatorially rigid in $Z$.  Consider a lift $\gamma_1:S^1\to Z_k\subseteq Z$.  A second 
such lift $\gamma_2$ contradicts malnormality immediately unless $j=k$, in which case it 
corresponds to an element $h\in F\smallsetminus G_j\langle g\rangle$ such that $hgh^{-1}\in G_i$ 
for some $i$, which also contradicts malnormality.

Conversely, suppose the $G_i$ are not malnormal, so that $g^h \in G_k$ and either $j\neq k$ or 
$h \in F-G_j$. Then the immersed loop $\gamma:S^1 \to X$ representing the conjugacy class of $g$ 
has two distinct lifts to $Z$, corresponding to the elements $g$ and $g^h$.
\end{proof}

\begin{lemma}\label{lem:random_is_malnormal}
Let $N$ be fixed and finite, and let $\phi_i:G=F_k \to F_l$ for $1\le i \le N$ be a finite 
collection of random homomorphisms  of length $n$.  Then the 
family of images $\{\phi_i(G)\}$ is malnormal  with probability $1-O(e^{-Cn})$.  
In particular, the same holds for a single random homomorphism.
\end{lemma}
\begin{proof}
Let $\sigma$ be any reduced word of length $n/2$ in $F_l$. If $w$ is a random reduced
word of length $n$, the probability that $w$ contains a copy of $\sigma$ is $O(e^{-Cn})$.
Moreover, if $w$ is a random reduced word of length $n$ conditioned to contain a copy
of $\sigma$, the probability that it contains more than one copy of $\sigma$ is $O(e^{-Cn})$.

The homomorphism $\phi_i$ assigns a random word of length $n$ to each edge of the rose $R_G$, 
and with probability $1-O(e^{-Cn})$ adjacent edges are folded at most $O(\log{n})$ in the 
corresponding folded graph $Y_{\phi_i}$.  For each edge $e$ of $R_G$, let $w_{e,i}$ be the 
subword of length $n/2$ in $Y_{\phi_i}$ starting at some fixed location in the interior. Then 
with probability $1-O(e^{-Cn})$,  for each $e$, this is the unique copy of $w_e$ in 
$\sqcup_i Y_{\phi_i}$. It follows that every nontrivial loop in 
$\sqcup_i Y_{\phi_i}$ is combinatorially rigid, so $\{\phi_i(G)\}$ is malnormal.
\end{proof}

Now let's suppose that we have chosen a finite collection of elements $g_i \in G$ 
whose union is homologically trivial in $G$ and such that each $g_i$ is fully combinatorially 
rigid. As a collection of cyclic words
in the generators $a,b,c,$ etc. of $F$, 
there are as many $a$s as $A$s, as many $b$s as $B$s and so on.

\begin{definition}
Let $S(Y)$ be a folded fatgraph whose boundary components $\partial_i S(Y)$ are
conjugate to $g_i \in G$ cyclically reduced.
The {\em $f$-vertices} on $\partial S(Y)$ are the vertices corresponding to the 
(valence $>2$) vertices of $Z$. We say $S(Y)$ is {\em $f$-folded} if every $f$-vertex 
maps to a $2$-valent vertex of $Y$, and distinct $f$-vertices map to distinct vertices of $Y$.
\end{definition}

Notice that we need the $g_i$ to be combinatorially rigid in order to unambiguously
identify where the $f$-vertices lie on each $\partial_i S(Y)$.

The $f$-folded condition just means that the vertices of $Y$ which have valence at least $3$
all correspond to interior points on edges of $Z_G$, after identifying $\partial S(Y)$ with
its image in $Z_G$ in the unique manner guaranteed by combinatorial rigidity.

\begin{proposition}\label{prop:f_folded_boundary_incompressible}
Let $S(Y)$ be $f$-folded. Then $S(Y)$ is boundary incompressible.
\end{proposition}
\begin{proof}
Let $\gamma_i:S^1\to Z$ be the geodesic representative of the conjugacy 
class of $g_i\in G$ and let $\hat{\gamma}_i:S^1\to \widehat{C}_G$ be the 
unique lift of $\gamma_i$.  By full combinatorial rigidity, the union 
of the lifts of the boundary components of $S(Y)$ to the fibre product
\[
S(Y)\times_{C_f} \widehat{C}_i\subseteq S(Y)\times \widehat{C}_i
\] 
is equal to the intersection 
$(S(Y)\times_{C_f} \widehat{C}_i)\cap(\partial S(Y)\times \coprod_i\hat{\gamma}_i(S^1))$.

Let $\widehat{S}(Y)$ be a component of some $S(Y)\times_{C_f} \widehat{C}_j$ that contains 
a lift $\hat{\delta}$ of some $\gamma_i$.  Then $\widehat{S}(Y)$ deformation retracts to a 
spine $\widehat{Y}$, which is a covering of $Y$.

The $f$-folded condition precisely means that, if $y$ is branch point of $Y$ 
and $z$ is a branch point of $Z$ that lies on some $\gamma_i$ then $\sigma(x)\neq f(y)$.

The claim now is that the core of $\widehat{Y}$ consists of precisely 
the lift $\hat{\delta}$.  Indeed, Stallings showed that the core of the fibre product 
is contained in the fibre product of the cores \cite[Theorem 5.5]{stallings_topology_1983}.  
Any branch vertex of the core of 
$\widehat{Y}$ contained in $\hat{\delta}$ maps on the one hand to a branch vertex 
of $Y$ and, on the other hand, to a branch vertex of $Z_G$ that lies on $\gamma_i$.  
This contradicts the $f$-folded hypothesis.

We have shown that every component of every $S(Y)\times_{C_f}\widehat{C}_j$ that 
contains a lift of a component of $\partial S(Y)$ has cyclic fundamental group.  
Therefore, $S(Y)$ is boundary-incompressible by Proposition \ref{prop:boundary_compressible_in_G}.
\end{proof}

\subsection{The Random Surface Subgroup Theorem}

We are now in a position to prove the main result in this section, the
Random Surface Subgroup Theorem. The most involved part of the argument is to show
that a pseudo-random homologically trivial chain bounds an $f$-folded surface; but actually,
this is already proved in \cite{Calegari_Walker_LP}, in the course of the proof of Thm.~8.9.

\begin{proposition}\label{prop:random_boundary_incompressible}
Let $\phi:F_k \to F_l$ be a random homomorphism of length $n$. Then with probability
$1-O(e^{-Cn})$ the image $\phi(F_k)$ is malnormal (so that every nontrivial element
is combinatorially rigid) and for every finite collection of nontrivial elements 
$g_i$ in $F_k$ whose image in $F_l$ is homologically trivial, there is an $f$-folded 
fatgraph $Y$ with $\partial S(Y)$ equal to the union of the $\phi(g_i)$.
\end{proposition}
\begin{proof}
The image is malnormal with the desired probability
by Lemma~\ref{lem:random_is_malnormal},
so it makes sense to talk about $f$-vertices on $\partial S(Y)$. By
Lemma~\ref{lem:random_is_pseudorandom}, for any fixed $T,\epsilon$, and for any finite
collection $g_i$, we can assume each individual $\phi(g_i)$ is $(T,\epsilon)$-pseudorandom,
also with the desired probability. 

Now, Calegari--Walker \cite[Theorem~8.9 (the Random $f$-folded Surface Theorem)]{Calegari_Walker_LP}, 
prove that for $(T,\epsilon)$ sufficiently big (depending on the rank $l$),
for any $(T,\epsilon)$-pseudorandom homologically trivial
immersed 1-manifold $\Gamma$ in $X$, and any subset $V$ of vertices in $\Gamma$
so that no two vertices in $V$ are closer than $N$ for some fixed $N\gg T$ (also
depending on $l$) there is a folded fatgraph $Y$ with $\partial S(Y) = \Gamma$
and such that every vertex in $V$ maps to a $2$-valent vertex of $Y$, with distinct
vertices of $V$ mapping to distinct vertices of $Y$. The first step of the construction
(in place of 8.3.2 in \cite{Calegari_Walker_LP})
is to take for each vertex $v \in V$, an arc $\sigma$ of length $2$ in $\Gamma$
containing $v$ as the midpoint, and glue it to a disjoint copy of $\sigma^{-1}$ in
$\Gamma$ not containing any point in $V$. By $(T,\epsilon)$-pseudorandomness, 
and the sparsity of $V$ in $\Gamma$, this pairing can be done, 
producing what is called in \cite{Calegari_Walker_LP} a
{\em chain with tags}. Then the remainder of the proof of \cite{Calegari_Walker_LP}
applies verbatim (actually, the proof is even easier than in \cite{Calegari_Walker_LP} since
the last step 8.3.7 is unnecessary).

In our context, taking $V$ to be the $f$-vertices on the $\phi(g_i)$, 
we have $N = O(n)$, and $S(Y)$ will be $f$-folded, as claimed.
\end{proof}

\begin{theorem}[Random Surface Subgroup]\label{thm:random_amalgam_surface}
Let $H:=F_1 *_G F_2$ or $H:=F *_G$ be obtained by amalgamating two free groups over random
subgroups of rank $k\ge 1$ of length $n$, or by taking an HNN extension over two random
subgroups of rank $k\ge 1$ of length $n$. Then $H$ contains a closed surface subgroup
with probability $1-O(e^{-Cn})$.
\end{theorem}
\begin{proof}
Let $g_i$ be any finite collection of nontrivial elements in $G$ which is homologically
trivial. For example, we could pick any nontrivial element $g\in G$ and take for our collection
the union of $g$ and $g^{-1}$.

By Proposition~\ref{prop:random_boundary_incompressible} for each of the inclusion maps
$\phi_j$ of the edge group, the image $\phi_j(g_i)$ bounds an $f$-folded surface. 
By Proposition~\ref{prop:f_folded_boundary_incompressible} these surfaces are
injective and boundary incompressible in their respective factors, so the closed surface
obtained by gluing them along their boundary is injective in $H$.
\end{proof}

\begin{remark}
We may weaken the hypotheses of Theorem~\ref{thm:random_amalgam_surface} in some
circumstances. If $G \to F$ is an inclusion of free groups, and $\ker:H_1(G) \to H_1(F)$
is nontrivial, \cite[Proposition~6.3 ]{Calegari_Walker_LP} says that we can find a
surface (group) in $F$ with boundary representing some nontrivial class in the kernel
which is an absolute minimizer for the scl norm (i.e.\/ relative 2-dimensional
Gromov norm). Such a minimizer is necessarily incompressible and boundary incompressible.
So if we build $H = F_1 *_G F_2$ (for example)
where $G \to F_1$ has $\ker:H_1(G) \to H_1(F_1)$ nontrivial, 
and $G \to F_2$ is random, then $H$ contains a closed surface subgroup, with probability
$1-O(e^{-Cn})$.
\end{remark}

\begin{remark}\label{remark:many_subgroups}
Notice that our argument gives rise to {\em many} surface subgroups; at least one for every
homologically trivial collection of conjugacy classes in $G$ (and actually many more than that, since
there are many choices in the construction of an $f$-folded surface). Hence the number of
surface subgroups of genus $g$ in a random graph of free groups should grow at least like
$g^{Cg}$. The homologically trivial collection of conjugacy classes may be recovered from
the surface by seeing how it splits in the graph-of-groups structure, so these surfaces are
really distinct.
\end{remark}

\section{Acknowledgments}

We would like to thank Fr\'ed\'eric Haglund and Alden Walker for useful conversations.
Danny Calegari was supported by NSF grant DMS 1005246.  Henry Wilton was supported by an EPSRC Career Acceleration Fellowship.

\appendix

\makeatletter
\@addtoreset{theorem}{section}
\makeatother

\renewcommand\thetheorem{\thesection.\arabic{theorem}}
\setcounter{theorem}{0}

\section{Graphs of free groups in virtually special groups}

In this appendix, we explain why many families of word-hyperbolic groups are known to contain 
one-ended fundamental groups of graphs of free groups with quasiconvex edge groups.

Recall that a group is \emph{special} if it is the fundamental group of a compact, non-positively 
curved, special cube complex in the sense of Haglund and Wise \cite{haglund_special_2008}.  
The reader is referred to that paper for the definition; we will only need the fact that the 
codimension-one hyperplanes of a special cube complex are embedded.  Agol proved that any 
word-hyperbolic group which is also the fundamental group of a non-positively curved cube complex 
is virtually special.  Hence, we have the following families of examples (among others).

\begin{example}[Random groups]
By a theorem of Dahmani--Guirardel--Przytycki \cite{dahmani_random_2011}, a random group (in the 
density model) is never special.  However, at densities less than $1/6$, a random group is 
cubulated \cite{Ollivier_Wise} and hence virtually special.
\end{example}

\begin{example}[Small-cancellation groups]
All $C'(1/6)$ groups are cubulated \cite{Wise_cancel} and hence virtually special.
\end{example}

\begin{definition}\label{definition:hyperplane}
Let $X$ be a (nonpositively curved) special cube complex and let $\Gamma=\pi_1(X)$.  
A \emph{codimension-1 hyperplane 
subgroup} of $\Gamma$ is the image of the fundamental group of a hyperplane under the map 
induced by inclusion.  Because hyperplanes are convex, the induced map is injective.  More 
generally, a  \emph{codimension-$(k+1)$ hyperplane subgroup} is the intersection of two 
codimension-$k$ hyperplane subgroups, where a base point is fixed on the intersection of 
the hyperplanes being considered.  Note that, if $X$ is compact, then there are only finitely 
many non-trivial codimension-$k$ subgroups as $k$ varies. 
\end{definition}

\begin{remark}
Because hyperplanes are embedded, each codimension-$k$ hyperplane subgroup is a graph of 
groups over the codimension-$(k+1)$ hyperplane subgroups it contains, with quasiconvex edge groups.
\end{remark}

The following is well known to the experts, but as far as we are aware does not appear in the 
literature.  We include it here for completeness.

\begin{theorem}\label{thm: GoGs inside special groups}
Let $X$ be a compact, non-positively curved cube complex in which every codimension-one hyperplane 
is embedded, and suppose that $\pi_1(X)$ is word-hyperbolic and one-ended.  Then $\pi_1(X)$ has a 
subgroup $H$ which is one-ended, word-hyperbolic and the fundamental group of a graph of 
free groups in which the edge groups are quasiconvex.
\end{theorem}

In particular, the conclusion of the theorem holds for virtually special groups.

To prove the theorem we make use of the following lemma, which is fundamental to the work of 
Diao and Feighn \cite[p. 1837]{diao_grushko_2005} and goes back to a theorem of 
Shenitzer \cite{shenitzer_decomposition_1955}.

\begin{lemma}\label{lem: Compatible free splitting}
Let $G$ be finitely generated and the fundamental group of a graph of groups $\mathcal{G}$ 
with non-trivial, finitely generated edge groups and suppose that $G$ splits freely.  
Then $G$ is the fundamental group of a graph of groups $\mathcal{G}'$ in which every edge group 
of $\mathcal{G}'$ is a finitely generated subgroup of an edge group of $\mathcal{G}$ and some 
vertex group of $\mathcal{G}'$ splits freely relative to the incident edge groups.
\end{lemma}
\begin{proof}
Let $T$ be the Bass--Serre tree for the graph of groups $\mathcal{G}$ and let $S$ be the 
Bass--Serre tree for some non-trivial free splitting of $G$.  Because $G$ is finitely generated, 
we may assume that the actions of $G$ on $S$ and $T$ are both cocompact.  Consider the 
diagonal action of $G$ on $S\times T$; the quotient $Q=(S\times T)/G$ naturally has the 
structure of a complex of groups with fundamental group $G$.  

Consider the action of $\mathcal{G}_v$, a vertex group of $\mathcal{G}$, on $S$.  
If $\mathcal{G}_v$ fixes a point then the conclusion of the lemma already holds.  
Therefore, we may assume that $\mathcal{G}_v$ does not fix a point.  
Because $\mathcal{G}_v$ is finitely generated (which follows from the fact that $G$ and the 
edge groups are finitely generated), there is a unique minimal, $\mathcal{G}_v$-invariant 
subtree $S_v$, on which $\mathcal{G}_v$ acts cocompactly.

Similarly, every edge group $\mathcal{G}_e$ of $\mathcal{G}$ acts cocompactly on a unique 
minimal invariant subtree $S_e\subseteq S$: either $\mathcal{G}_e$ fixes a vertex of $S$, 
which must be unique because $\mathcal{G}_e$ is non-trivial but edge stabilizers of $S$ are trivial, 
or otherwise $S_e$ exists because $\mathcal{G}_e$ is finitely generated.

Now, $Q$ has a compact core $K$, which can be described as follows.  Consider the subcomplex 
$\widetilde{K}\subseteq S\times T$ consisting of all pairs $(s,t)$ such that $s\in S_x$ where 
$x$ is the vertex or edge of $T$ that contains $t$.  Then $\widetilde{K}$ is a contractible, 
$G$-invariant subcomplex of $S\times T$ and the quotient $K=\widetilde{K}/G$ is the 
required compact core.

The complex $K$ is a square complex; call the 1-cells that are images of edges of $S$ horizontal 
and the images of edges $T$ vertical.   The edge stabilizers of the action of $G$ on $S$ are 
realized by horizontal subgraphs contained in the middle of the squares of $K$.  Because the 
edge stabilizers of $S$ are trivial these subgraphs are trees, and a square of $K$ containing 
a leaf can be collapsed onto three of its sides.  Proceeding inductively, we may collapse 
such a tree to a single point $y$ in a subcomplex $K'$.

Cutting vertically down the middles of the squares of $K'$ gives a new graph of groups 
$\mathcal{G}'$ for $G$ in which the edge groups are subgroups of the edge groups of $\mathcal{G}$.  
The vertical component of the 1-skeleton of $K'$ that contains $y$ is a vertex group of $\mathcal{G}'$ 
that splits freely relative to the incident edge groups, as required.
\end{proof}

Combining Lemma \ref{lem: Compatible free splitting} with Grushko's Theorem, we quickly obtain 
the following.

\begin{lemma}\label{lem: One-ended factor}
Suppose $G$ is word-hyperbolic and the fundamental group of a graph of free groups in which 
every edge group is quasiconvex. Then 
\[
G\cong F_r*G'
\]
where $F_r$ is a free group and $G'$ does not split freely and is the fundamental group of 
a graph of free groups in which every edge group is quasiconvex.
\end{lemma}
\begin{proof}
By Lemma \ref{lem: Compatible free splitting} and induction (invoking Grushko's theorem), we have that 
\[
G\cong F_r*G'
\]
where $G'$ is the fundamental group of a graph of free groups in which every edge group is 
a finitely generated subgroup of an edge group of $G$. Necessarily, $G'$ is word-hyperbolic 
and the edge groups of $G'$ are quasiconvex because free groups are locally quasiconvex.  
\end{proof}

\begin{proof}[Proof of Theorem \ref{thm: GoGs inside special groups}]
Consider a descending chain of subgroups
\[
\Gamma=H_0\supseteq H_1\supseteq H_2\supseteq\ldots \supseteq H_k\supseteq\ldots
\]
where $H_k$ is a codimension-$k$ hyperplane subgroup.  Let $k$ be maximal such that $H_k$ is 
non-free.  Then $H_k$ is word-hyperbolic, non-free and the fundamental group of a graph of 
free groups with quasiconvex edge groups.  By Lemma \ref{lem: One-ended factor}, we can write
\[
H_k=F_r*H
\]
where $H$ is as required; note that $H$ is one-ended because $H_k$ is torsion-free and non-free.
\end{proof}

\end{document}